\documentclass{amsart}
\usepackage{vmargin, amsmath,amsfonts,amssymb,amsthm,xcolor,amscd,latexsym,xspace,enumerate,bigints}
\usepackage[pagebackref, colorlinks=true,linkcolor=blue,citecolor=red]{hyperref}
\usepackage{hyperref}
\usepackage{pdfpages}
\usepackage{tikz-cd}
\usetikzlibrary{arrows}
\usepackage{graphicx}
\usepackage{mathtools}
\usepackage{stmaryrd}
\usepackage{stackrel}
\usepackage{array,booktabs}
\usepackage{booktabs,tabularx}
\usepackage{soul}
\sethlcolor{green}
\usepackage{wrapfig}
\usepackage{graphicx}
\usepackage{changepage}%for indents
\usepackage{bigints}
\usepackage[normalem]{ulem}

\graphicspath{ {./images/} }
\makeindex

\newtheorem{theorem}{Theorem}[section]

\newtheorem{lemma}[theorem]{Lemma}
\newtheorem{proposition}[theorem]{Proposition}
\theoremstyle{definition}
\newtheorem{definition}[theorem]{Definition}
\newtheorem{remark}[theorem]{Remark}

\newtheorem{example}[theorem]{Example}
\numberwithin{equation}{section}

\newcommand{\cgk}{C(G,\mathbb{K})}
\newcommand{\rgk}{R(G,\mathbb{K})}
\newcommand{\lgk}{L^2 (G,\mathbb{K})}
\newcommand{\lck}{L^2 (G,\mathbb{K})}
\newcommand{\cck}{C_C (G,\mathbb{K})}

\newcommand{\rhgk}{\hat{R}(G,\K)}

\newcommand{\C}{\mathbb{C}}
\newcommand{\R}{\mathbb{R}}
\newcommand{\N}{\mathbb{N}}
\newcommand{\Z}{\mathbb{Z}}
\newcommand{\T}{\mathbb{T}}
\newcommand{\K}{\mathbb{K}}

\newcommand*{\ix}[1]{\textit{#1}}
\newcommand{\n}{\hfill\\}

\newcommand*{\Tw}{\mathcal{T}}

\newcommand{\supp}{\text{supp}}
\newcommand{\Qp}{\mathbb{Q}_p}
\newcommand{\Zp}{\mathbb{Z}_p}

\begin{document}

\title[A Note on the Peter-Weyl Theorem]{A Note on the  Peter-Weyl Theorem}

\author[Y. Bavuma]{Yanga Bavuma}
\address{Yanga Bavuma\endgraf
Department of Mathematics and Applied Mathematics\endgraf
University of Cape Town\endgraf
Private Bag X1, Rondebosch, 7701, Cape Town, South Africa\endgraf
Email: \texttt{yanga.bavuma@uct.ac.za} -- \href{https://orcid.org/0000-0001-8991-6071}{ORCID: 0000-0001-8991-6071}}

\author[F.G. Russo]{Francesco G. Russo}
\address{Francesco G. Russo\endgraf
School of Science and Technology\endgraf
University of Camerino\endgraf
via Madonna delle Carceri 9, 62032, Camerino, Italy\endgraf
Email: \texttt{francesco.russo@unicam.it} -- \href{https://orcid.org/0000-0002-5889-783X}{ORCID: 0000-0002-5889-783X}}

\author[E. Stevenson]{Elizabeth Stevenson}
\address{Elizabeth Stevenson\endgraf
Department of Mathematics and Applied Mathematics\endgraf
University of Cape Town\endgraf
Private Bag X1, Rondebosch, 7701, Cape Town, South Africa\endgraf
Email: \texttt{stveli004@myuct.ac.za} -- \href{https://orcid.org/0009-0002-2821-0221}{ORCID: 0009-0002-2821-0221}}

\maketitle

\begin{abstract}We introduce some classical concepts in the representation theory of compact groups, in order to  use them for a new generalization of the Peter-Weyl Theorem. We mostly deal with functions on locally compact groups possessing large nontrivial compact open subgroups: in fact, we show that these functions can be approximated via others which are locally identical to the well known representative functions.
\end{abstract}

\maketitle

\section{Introduction}\label{1}

In 1807, Joseph Fourier found a way to approximate arbitrary periodic functions using simple linear combinations of sines and cosines. His ideas were a pioneering work for that time and led to what is known as the Fourier Theorem today. Then, 120 years later, in a German paper whose title may be translated to mean "The completeness of primitive representations of a closed continuous group", Fritz Peter and Hermann Weyl \cite{pw} published the celebrated "Peter-Weyl Theorem", a milestone in the representation theory of compact groups. This result is much like the Fourier Theorem as it is also concerned with approximating continuous functions using simpler functions. However, instead of real valued periodic functions, the Peter-Weyl Theorem approximates functions defined on certain structures of a geometric nature called \textit{compact groups}, which are groups that have been given with operations which are compatible with the topology under which they are compact and Hausdorff, see \cite{HewittRoss, hilgertneeb, hm, stroppel}. In other words, Peter-Weyl Theorem shows that it is possible to approximate any continuous function on a compact group with what are called \textit{representative functions}, which are much simpler functions than the original ones. In the specific case of  connected compact groups sharing the geometric properties of the circle of the usual Euclidean plane $\R^2$ (i.e. the \textit{torus} $\T$, see \cite{hm}), the Fourier Theorem becomes a corollary of Peter-Weyl's Theorem. The unpublished MSc thesis of the third author \cite{ste}  illustrates a full discussion around these two results in a self-contained way. In fact, these two important results (the Peter-Weyl Theorem and the Fourier Theorem) have been proved in many different ways in the literature. For example in \cite[\S 11.5]{lda} there is an argument involving Følner sets and Følner sequences on arbitrary compact abelian groups. Fourier's original argument (see \cite[Section 1.5]{enrique}) is also completely different and works only on the torus $\mathbb{T}$. Different authors \cite{FMS} follow an approach (always on $\mathbb{T}$) which uses testing polynomial families of possibly various natures (compare with \cite[4.26 Fourier Series]{rudin}). 

In order to discuss the Peter-Weyl Theorem we first need to introduce some general notions about topological algebra, as well as define representations and actions of compact groups on topological vector spaces. Briefly, what we mean by a representation is a group action of a topological group on a topological space, which is continuous with respect to both topologies. This is a standard concept which follows a classical line in topological group theory, see \cite{HewittRoss, hilgertneeb, hm, hofmann, stroppel}.

The idea of the generalization we will present here is that given a topological group with a nontrivial compact open subgroup, it is possible to partition the group into compact open cosets of this subgroup. And since each of these cosets is topologically isomorphic to a compact group we may in a way use the Peter-Weyl Theorem on each of the cosets to locally approximate sections of a given continuous function on the original group. These approximations are then glued together to approximate the full function. Of course, our main result specializes to the Peter-Weyl Theorem when the compact open subgroup is the group itself.  After the current Section \ref{1} which serves as introduction, Section \ref{2} is dedicated to stating the necessary preliminaries, covering topics such as representation theory, Haar measures and the original formulation of the Peter-Weyl Theorem. 
In Section \ref{3} we prove our main result which shows an approximation analogous to what is given by the representative functions $\rgk$ in the Peter-Weyl theorem  which may be produced for scalar valued functions defined on locally compact groups with compact open subgroups. An example of an applicable locally compact group is given by the $p$-adic rationals $\mathbb{Q}_p$, which contains the $p$-adic integers $\mathbb{Z}_p$ as a compact open subgroup. We will discuss this example as well.
The notation we use is standard and follows mostly \cite{HHR, HewittRoss, hilgertneeb,  hm, hofmann, stroppel}, which are classical textbooks in topological group theory and abstract harmonic analysis. 

\section{General notions of topological algebra}\label{2}

The present section recalls some unpublished material from the MSc thesis of the third author \cite{ste} written under the supervision of the first two authors. Throughout this section, we consider $\K$ to be equal to $\C$ (the field of complex numbers) or $\R$ (the field of real numbers). First we will discuss topological groups and representations. Throughout this paper we shall refer to a \textit{topological group} as a (Hausdorff) group with a topological structure on it. In order to match the algebraic structure and the topological structure in an appropriate way, we require that the group operation and the inversion operation be continuous with respect to the topology on the group, see \cite[Definition 1.1]{hofmann}.  If a topological group is compact under its topology, we say that it is a \textit{compact group}. More generally, a \textit{locally compact group} is a topological group that is locally compact under its topology; see \cite{HHR, hofmann, stroppel, willard}. Similarly, we may topologize the underlying set of a vector space instead of the underlying set of a group. When we do this we require a topology under which the main operations of the vector space are continuous. That is: the operations of addition, additive inversion, and scalar multiplication. For instance, following \cite[Chapter 2]{hofmann} suppose $E$ is a vector space on $\K$ with topology $\mathcal{T}_E$ and $\mathcal{T}_{\K}$ respectively. Let $\mathcal{T}_{E \times E}$ be the product topology of $E \times E$ and let $\mathcal{T}_{\K \times E}$ be the product topology on $\K \times E$. We say that $E$ is a \textit{topological vector space}, if the addition $(u_1,u_2) \in E \times E \mapsto u_1 + u_2 \in E$ is continuous w.r.t. $\mathcal{T}_{E \times E}$ and $\mathcal{T}_E$; the inversion $u \in E \mapsto -u \in E$ is continuous w.r.t. $\mathcal{T}_E$;
and the scalar multiplication $(c,u) \in \K \times E \mapsto c u \in E$ is continuous w.r.t. $\mathcal{T}_{\K \times E}$ and $\mathcal{T}_E$. Recall the notion of a group action: elements of a group can be seen to ``act" on elements of another set, producing again an element of the given set. For instance, by shifting or permuting its elements, we have examples of actions. This is a well known notion in topological group theory, see \cite{HHR, hilgertneeb, hofmann,  stroppel}.
When this holds, we may introduce the following concept:

\begin{definition}[$G$-Module, See \cite{hofmann}, Definition 2.1 (i)]\label{gmod}Let $G$ be a topological group with topology $\mathcal{T}_G$. Let $E$ be a topological vector space with topology $\mathcal{T}_E$. Let $(g,t)\in G\times E \mapsto gt\in E$ be a group action of $G$ on $E$.
We say $E$ is a \ix{$G$-module} with respect to the above action if
\begin{enumerate}
\item[{\rm (i).}] For all $g\in G$, the map $t\in E\mapsto gt\in E$ is continuous and linear w.r.t. $\mathcal{T}_E$, and
\item[{\rm (ii).}] For all $t\in E$, the map $g\in G\mapsto gt\in E$ is continuous w.r.t. $\mathcal{T}_G$ and $\mathcal{T}_E$.
\end{enumerate}
\end{definition}

The definition of a \textit{linear function} in the context of vector spaces is well known: essentially we consider a function $f$ from a vector space $X$ onto a vector space $Y$ (both over the field $\mathbb{K}$) and say that f is linear if $f({x}+{y})=f({x})+f({y})$
and $f(c{x})=cf({x})$ for all ${x}, {y} \in X$ and $c \in \mathbb{K}$.
It is more relevant to recall the following notion:

\begin{definition}[Strong Operator Topology and Representations, See \cite{hofmann}, Chapter 2]\label{defrepresentation} Let $E$ be a topological vector space with topology $\Tw_E$. Let $\Tw_{E^E}$ denote the product topology on $E^E$. The set of continuous linear functions $E\rightarrow E$ given by $\mathrm{Hom}(E,E)\subseteq E^E$ forms a topological vector space with respect to the induced topology of $\Tw_{E^E}$. This topological space is denoted by \ix{$\mathcal{L}_p(E)$} and its topology is called \textit{strong operator topology}.
Let $G$ be a topological group with topology $\Tw_G$. Consider a map $\pi: G\rightarrow \mathcal{L}_p(E)$. We say that $\pi$ is a \ix{representation} of $G$ on $E$ if
\begin{enumerate}
\item[{\rm (i)}.] $\pi$ is continuous w.r.t. $\Tw_G$ and the topology of $\mathcal{L}_p(E)$,
\item[{\rm (ii)}.] $\pi$ is a homomorphism of groups.
\end{enumerate}
\end{definition}

We can also summarize the above two conditions of Definition \ref{defrepresentation}, saying that $\pi$ is a morphism in the category of topological groups. When we deal with topological groups, we always look for those homomorphisms which preserve both the topological and algebraic structure. It turns out that the ideas of $G$-modules and representations are equivalent, where a given $G$-module may produce a representation by defining $\pi(g)(x)\coloneqq g\cdot x$.

An important example of a $G$-module is given by the topological vector space $$\cgk :=\{ f : G \to \K \mid f \ \mbox{is continuous}\},$$ consisting of all continuous functions from $G$ to $\K$ w.r.t. the topology on $G$ and the usual topology in $\K$. The group action of an element $g$ on a function $f:G\rightarrow\K$ is obtained by composing $f$ with the right multiplication map $h\in G\mapsto hg\in G$. We consider this action in more detail below.

\begin{example}\label{supnorm} Following \cite{HewittRoss, hilgertneeb, hofmann, stroppel}, if $G$ is a compact group, it turns out that $\cgk$ is a topological vector space under the topology induced by the \textit{norm of the supremum}
\[
||f||_\infty=\sup_{g\in G}f(g).
\]
There exists an action of $G$ on $\cgk$, given by
$$ (g,f)\in G\times\cgk\mapsto {}_gf \in \cgk,$$
where for each $g\in G$ we have \[_gf(x)\coloneqq f(hg) \ \ \mbox{for all} \ \  x\in G.\] It holds that $\cgk$ is a $G$-module in the present circumstances.
\end{example}

Note that $\cgk$ is both a topological vector space and a \textit{Banach algebra}, see \cite[Chapter 1]{hofmann}.   In fact, for any compact group $G$, one can see that $\cgk$ with the norm of the supremum is a Banach space and is, in fact, a Banach $G$-module.   Next, we must discuss the Haar measures, which we will use to measure the sizes of scalar-valued functions defined on locally compact groups. 

\begin{definition}[Support, $C_c(G,\K)$ and $C_c^+(G,\K)$, See \cite{stroppel}, Definition 12.1]\label{closedsupport}
We define the \textit{support} of $\varphi \in C(G,\mathbb{K})$, where $G$ is a topological group, as $$\supp (\varphi) := \overline{\{g \in G \ | \ \varphi(g) \neq 0\}}.$$ Then we define
 $C_c(G,\K)$ to be the collection of continuous maps from $G$ to $\K$ with compact support. 
 And $C_c^+(G,\K)$ is defined to be the collection of continuous maps $\varphi\in C_c(G,\K)$ where $\varphi(g)\geq 0$ for all $g\in G$.
\end{definition}

A \textit{Haar measure} is a positive, linear, invariant, nonzero map that takes scalar-valued, continuous, compactly supported functions on a topological group $G$ (that is, functions in $\cck$) to scalar values, see \cite[Definitions 12.3, 14.2]{stroppel}.
An unfortunate aspect of Haar measures is that a single locally compact group $G$ might have multiple Haar measures, which are all proportional to each other. For instance if $\mu$ is a Haar measure on $G$, then $2\mu$ and $3\mu$ or in general $p\mu$ for any positive $p\in \R$ will also be Haar measures. Luckily, this is the only way to produce additional Haar measures, as they are unique up to a constant.

\begin{lemma}[Haar Measures for Locally Compact Groups, See \cite{stroppel}, Theorems 12.20, 12.23 and 14.3]\label{resultofexistence} Let $G$ be a locally compact group.
 Then there exists a map $\lambda:C_c(G,\K)\rightarrow\K$ which is a Haar measure. Moreover, fix $\eta\in C_c^+(G,\K)$ where $\eta\neq0$, and
 suppose $\mu:C_c(G,\K)\rightarrow\K$ is another Haar measure. Then, for all $\varphi\in C_c(G,\K)$ we have
 $$ \frac{\lambda(\varphi)}{\lambda(\eta)}=\frac{\mu(\varphi)}{\mu(\eta)}.
$$ Therefore there exists an $r \in ]0,+\infty[$ such that $\lambda=r\mu$.
\end{lemma}

This uniqueness (up to the positive constant $r$) allows us to introduce the following notation.

\begin{definition}[Haar Measure Notation, See \cite{stroppel} Theorem 12.24] Let $G$ be a locally compact group.  Suppose $\lambda:C_c(G,\K)\rightarrow\R$ is a Haar Measure and $\varphi\in C_c(G,\K)$.  We write
 \[
 \lambda(\varphi)= \int_G \varphi \text{ }d\lambda=\int \varphi \text{ }d\lambda
 \]
\end{definition}

When  $G$ is a compact group, we may ``normalize'' the Haar measure by taking any Haar Measure $\lambda$ and dividing it by its measure of the unit function $\textbf{1}:G\rightarrow\K$ given by $\textbf{1}(x)=1$ for all $x \in G$. 

\begin{definition}[Normalized Haar Measure, See \cite{hofmann}, Definition 2.6]\label{defnormhaarmeasure}Let $G$ be a compact group and $\mu:\cgk\rightarrow\K$ a Haar measure. If $\mu(\textbf{1})=1$,  we say that $\mu$ is a \textit{normalized Haar measure}.
\end{definition}

Unfortunately the Haar measure may not be normalized in this way when we are working with a locally compact noncompact group, as the constant function \textbf{1} may fail to have compact support. 

\begin{proposition}[Existence and Uniqueness of Normalized Haar Measures, See \cite{hofmann}, Theorem 2.8] Given a compact group $G$, there exists one, and only one, normalized Haar measure on $G$.
\end{proposition}

Note that we may define the well known $L^2$-norm on $C_c(G,\K)$. The completion of $C_c(G,\K)$ under  the $L^2$-norm is $\lgk$ if $G$ is compact. 
\begin{remark}\label{l2definition} We recall some facts from \cite[Example 2.12]{hofmann}, \cite[Chapter 3]{rudin} and \cite[Lemma 14.6]{stroppel}.  For any locally compact group $G$, we may define the following scalar product on $\cck$ with respect to a prescribed Haar measure $\mu$
 \[
 (f_1\mid f_2):=\int_G{f_1(g)\overline{f_2(g)}}{ d \mu}.		
 \] 
 The induced norm  by this scalar product is  the $L^2$-norm, and is given explicitly (for any $f\in\cck$) by
 \[
 ||f||_2=\sqrt{\int_G{f(g)\overline{f(g)}}{ d \mu}}=\sqrt{\int_G{|f(g)|^2}{ d \mu}}.
 \]
 The completion of both $\cck$ (in the case where $G$ is locally compact) and $\cgk$ (in the case where $G$ is compact) under this norm give rise to $\lck$.
\end{remark}

We are now almost ready to state the Peter-Weyl Theorem. 

\begin{definition}[Almost Invariant Elements, See \cite{hofmann}, Definition 3.1]Let $G$ be a locally compact group, $E$ a $G$-module and $x\in E$. We say that $x$ is \textit{almost invariant} if $\mathrm{span}\{G\cdot x\}$ is a finite-dimensional vector space. \end{definition}

The almost-invariant elements of $\cgk$ will be very relevant in our discussions.

\begin{definition}[Representative Functions, See \cite{hofmann}, Definition 3.3]\label{defrepresentativefunctions}Let $G$ be a compact group. The set of almost invariant functions in $\cgk$ is denoted by $\rgk$. That is,
\[\rgk=\{f:G\rightarrow\K \mid \mathrm{span}\{G\cdot f\} \mbox{ is finite-dimensional}\}\]
This is  the \textit{space of representative functions}. An element of $\rgk$ is  a \textit{representative function}.
\end{definition}

 The classical formulation of the Peter-Weyl Theorem is the following:

\begin{lemma}[The Peter-Weyl Theorem, See \cite{hofmann}, Theorem 3.7]\label{smallpw}
Consider a compact group $G$ and the Banach algebras $\rgk$, $\cgk$ and $\lgk$ as in Remark~\ref{l2definition}. Then $R(G,\K)$ is a dense Banach subalgebra in $\cgk$ and in $\lgk$.
\end{lemma}

Lemma \ref{smallpw} is also known as the ``Small Peter-Weyl Theorem'', or the ``Classic Peter-Weyl Theorem'', in order to differentiate it with the ``Big Peter-Weyl Theorem'', which is an extension of the Peter-Weyl Theorem. With the "Big Peter-Weyl Theorem", one can work with a compact group $G$, but we replace the Banach algebra $\rgk$ with a more general topological vector space $E$, such as the $G$-complete locally convex $G$-modules, see \cite[Theorem 3.51]{hofmann}.

 \section{What happens to locally compact groups with compact open subgroups ?}\label{3}

The goal of the result included in this section is to extend Lemma \ref{smallpw} to locally compact groups of a certain structure. Where the Peter-Weyl theorem uses compact groups, this result uses locally compact groups with  compact open subgroups (for instance, the $p$-adic rationals $\mathbb{Q}_p$ which have the $p$-adic integers $\mathbb{Z}_p$  as a compact open subgroup). While the Peter-Weyl Theorem is able to approximate all the elements of $\lgk$ (when $G$ is compact) with respect to the $L^2$-norm (defined via a normalised Haar measure), we are going to approximate all the elements  of $\lck$ under the $L^2$-norm when $G$ will be locally compact. Finally, where the Peter-Weyl Theorem is able to obtain its approximations using representative functions, we will approximate its functions using the so-called ``Lifted representative functions'' (see later on). To produce a lifted representative function we first need a locally compact group $G$ with an open compact subgroup $H$. We then take a representative function on the compact open subgroup $H$ and "lift" it to be a function on the including group $G$. The lifted function will take its original values on the subgroup $H$ and have the value 0 outside of it. Shifts and linear combinations of such functions are also considered to be lifted representative functions. Before we define the lifted representative functions it is useful to define the "lifting operator" which is an operator that takes a function defined on the subgroup $H$ and lifts it to a function on the $G$ which takes the value zero outside of $H$. The lifting process is demonstrated in Fig. \ref{SingleLiftedRep}.

\begin{definition}[Lifting Operator]\label{liftingop} Let $G$ be   a locally compact group with  $H$ a compact open subgroup of $G$. Consider a continuous function $k : x\in H \mapsto k(x) \in \K$.  We define the \textit{lifting} of $k$ to $G$ as the function $\ell_G(k):G\rightarrow\K$ given by
 \[
 \ell_G(k)(x)=\begin{cases}
 k(x) & \mbox{if} \ x\in H\\
 \\
 0 & \mbox{if} \  x\in G\setminus H.
 \end{cases}
\]
\end{definition}

\begin{figure}[!htbp]
 \centering
 \includegraphics[width=0.69\linewidth]{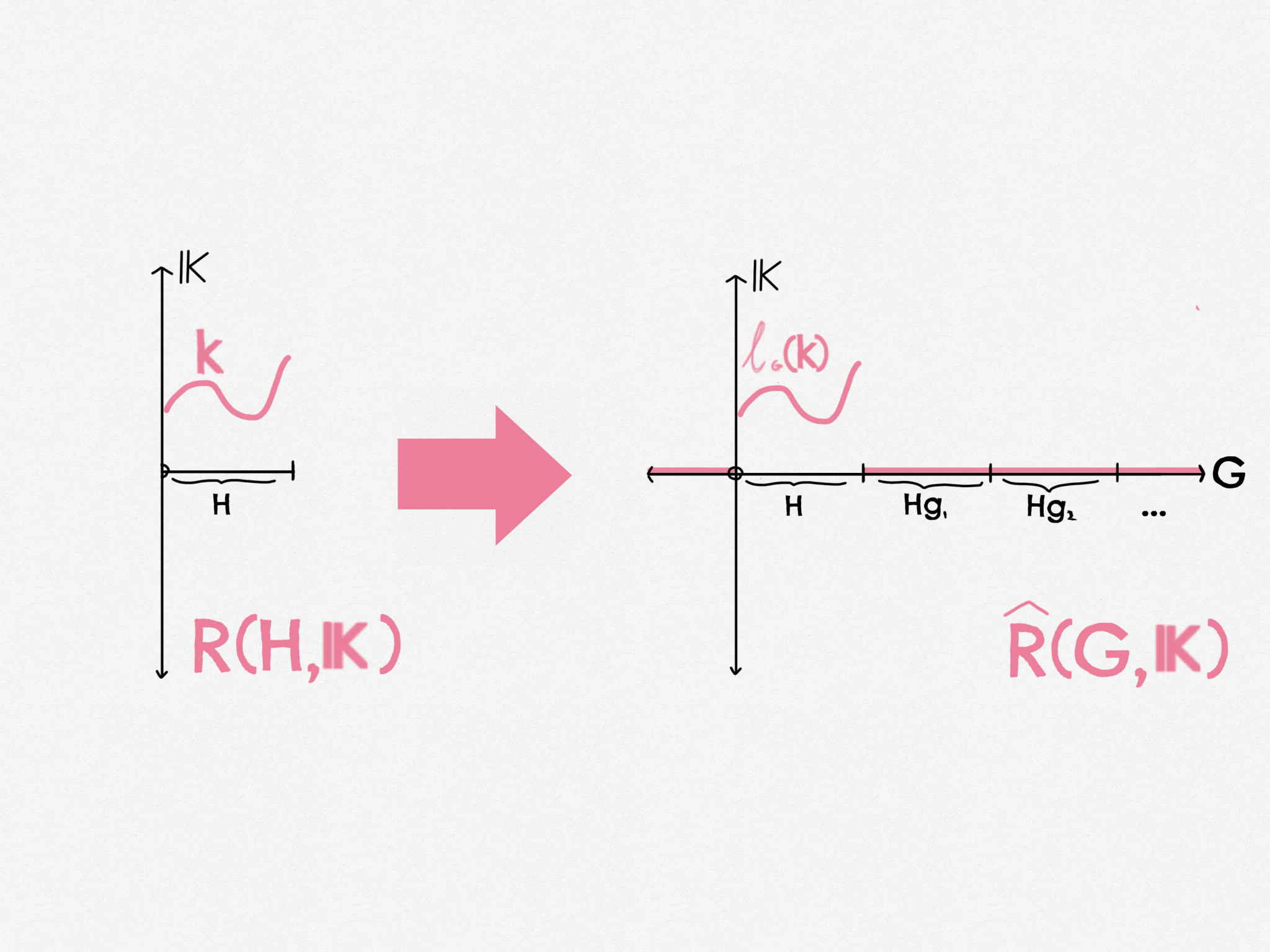}
 \caption{Demonstration of how a representative function may be lifted. Note that $G$ and $\K$ have been drawn like straight lines to give intuition, but will not in general have this structure.}
 \label{SingleLiftedRep}
\end{figure}

Note that we are not requiring that $\ell_G(k)$ is  continuous in Definition \ref{liftingop}, but this follows from an application of the Remark below (because $k(x)$ continuous on $H$).

\begin{remark} \label{lemmapointwise}  From  \cite[Theorem 7.6]{willard}, we know that for a locally compact  group $G=A \cup B$ with $A$ and $B$  both open (or both closed) subsets of $G$, if $f : G \to \K$ is a function such that both $f_{|A}$ and $f_{|B}$ are continuous,  then $f$ is continuous. This means that

Definition \ref{liftingop} gives us that $\ell_G(k)$ is always continuous, as $k(x)$ is continuous on $H$. This is because we may choose $H=A$ and $G\setminus H=B$. Here the assumption of $G$ being  a topological group is important, because in a topological group every open subgroup (such as $H$) is automatically a closed subgroup. This is  generally false for arbitrary open and closed sets of an arbitrary topological space. See \cite[Proposition 1.10 iii)]{hofmann} for details. \end{remark}

Lifted representative functions are defined in the way one might expect given how we have defined the lifting operator. When a locally compact group $G$ has a compact open subgroup $H$, we may lift an element of $R(H,\K)$ to $C(G,\K)$  to obtain bigger lifted representative functions. In fact, from Definition \ref{closedsupport}, we always consider the support to be a closed subset (hence compact since we are in Hausdorff spaces) for the functions which are in $\hat{R}(H,\K)$  below. This justifies $\hat{R}(G,\K) \subseteq C_c(G,\K)$. Note also that  shifts and linear combinations are  representative functions.

\begin{definition}[Lifted Representative Functions]\label{liftedrep} If  $G$ is a locally compact group and  $H$ a compact open  subgroup of $G$,  we may define the collection of lifted representative functions to be the set
 \[
 \hat{R}(G,\K)\coloneqq \text{span}\left[\{{}_g\ell_G({k}) \mid k \in R(H,\K), \ \ g \in G \}\right],
\]
 where \[{}_g\ell_G({k})(x)\coloneqq \ell_G({k})(xg)=\begin{cases}
    k(xg) & \mbox{if} \  x\in Hg^{-1}\\
    \\
    0 & \mbox{if} \  x\in G\setminus Hg^{-1}     
 \end{cases}\] 
 denotes the $g$-shift of $\ell_G({k})$. 
\end{definition}

In Fig. \ref{CompositeLiftedRep} we can see how a lifted representative function on $G$ can be formed by combining one or more existing representative functions on $H$.
\begin{figure}[!htbp]
 \centering
 \includegraphics[width=0.79\linewidth]{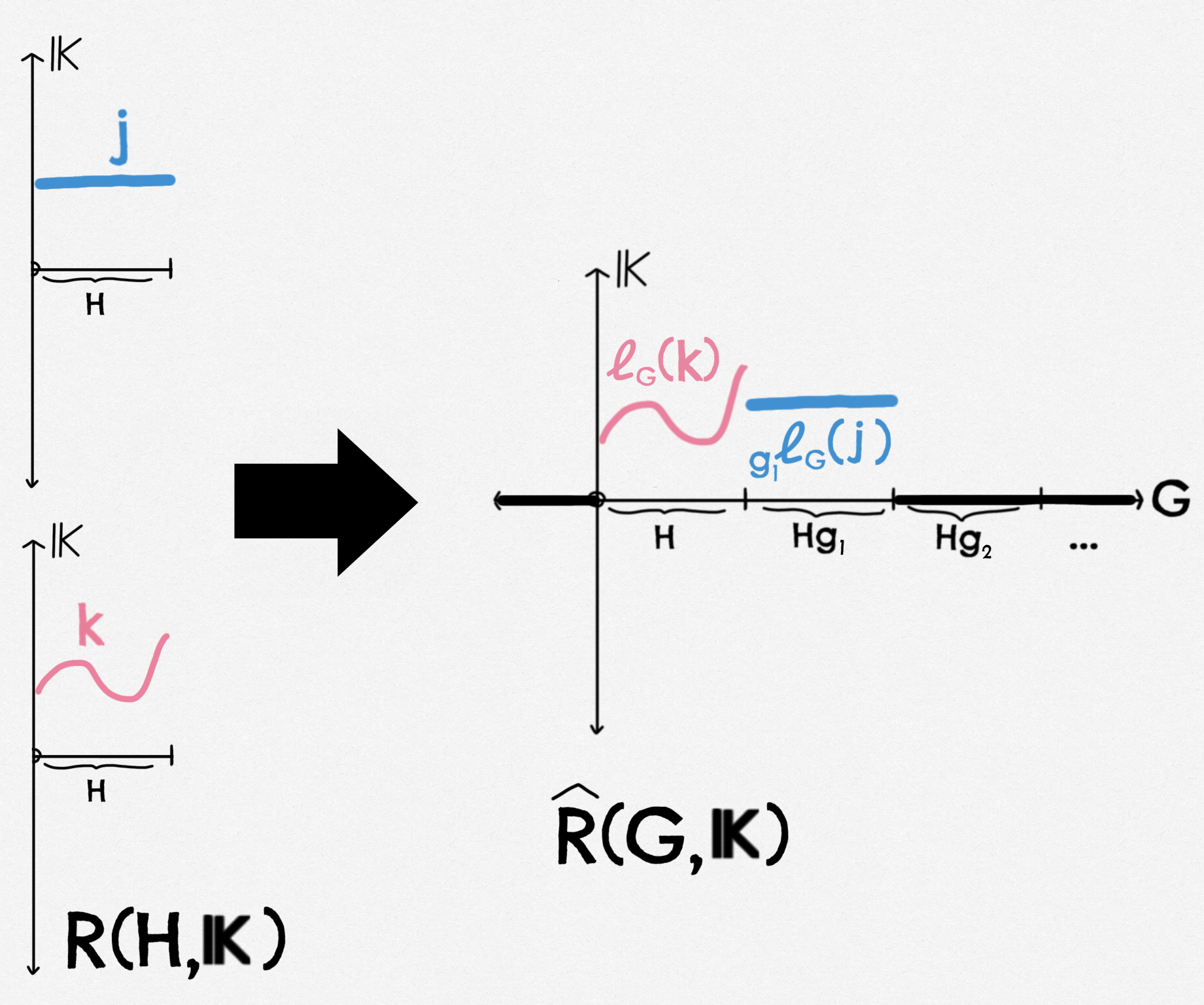}
 \caption{Demonstration of how a lifted representative function can be obtained using multiple representative functions. 
 Note that $G$ and $\K$ have been drawn like straight lines to give intuition, but will not in general have this structure.
 }
 \label{CompositeLiftedRep}
\end{figure}

\begin{remark}[Properties of the Lifting Operator]\label{liftingopproperties} Below we will mention a few key properties of the lifting operator that we will need for our deductions later on. Note we show many of these facts only for the $x\in H$ case as the $x\notin H$ case is trivial.
 \begin{enumerate}
 \item\textbf{The Lifting Operator is Linear:}
For all $k_1,k_2\in R(H,\K)$ and $a,b\in\K$ and for all $x\in H$
\[  \ell_G(ak_1+bk_2)(x)=a\ell_G(k_1)(x)+b\ell_G(k_2)(x). \]
 \item\textbf{The Lifting Operator Preserves Positivity:}
 Let $k\in R(H,\K)$ where $k  >0$ for all $x \in H$.  Then
\[ \ell_G(k)(x)=  k(x) \ge 0. \]
 \item\textbf{The Lifting Operator Preserves Shifts:}
 If $k \in R(H,\K)$ and $h\in H$, then for all $x\in H$
\[ \ell_G({}_h k)(x)= \ell_G(k)(xh)=  {}_h \ell_G(k)(x),\]
 therefore the shifting action of $H$ is preserved. %\textcolor{red}{(Questions over this one.)}. 
 \item\textbf{The Lifting Operator Preserves Multiplication:}
 If $k_1,k_2\in R(H,\K)$, then for all $x\in H$ we have
\[ \ell_G(k_1k_2)(x)=\ell_G(k_1)(x)\ell_G(k_2)(x). \] %\textcolor{red}{(Questions over this one.)}
 \item\textbf{The Lifting Operator Preserves the Norm of $\K$:}
 If $k\in R(H,\K)$, then for all $x\in H$
\[ \ell_G(|k|)(x) =|\ell_G(k)(x)|.\]
 \end{enumerate}
\end{remark}

An important fact that we will require for our result later is that when we lift a function from $H$ to $G$, the $L^2$-norm of the original function with respect to the Haar measure on $H$ is the same as the $L^2$-norm of the lifted function, provided we use a Haar measure on $G$ that is in a certain sense normalized over $H$.

\begin{lemma}[Preservation of Lifting $L^2$-Norms]\label{liftinglemma} Assume that $G$ is a locally compact group with  a nontrivial  compact open  subgroup $H$, and that $\int_G\cdot \;d\lambda$ is a Haar measure on $G$ satisfying (given $\chi_H:G\rightarrow\K$ characteristic function on $H$) $\int_G \chi_H  \; d\lambda=1.$ 
If  $||f||_2$ with $f\in\cck$ is the $L^2$-norm on $\lck$ and  $||k||_{2,H}$ is the $L^2$-norm on $L^2(H,\K)$ with $k\in C(H,\K)$, then  for all $j\in R(H,\K)$ 
 \[
 ||\ell_G(j)||_2=||j||_{2,H}.
 \]
\end{lemma}
\begin{proof}

 First, we claim that the function $\rho:C(H,\K)\rightarrow\K$ given by $\rho (j)\coloneqq \int \ell_G({j})\;d\lambda$ is the normalized Haar measure on $H$. We prove it, checking the definitions and utilizing Remark \ref{liftingopproperties}.  Fix $j,j_1,j_2\in R(H,\K)$ for the below:
 \begin{enumerate}
 \item \textbf{Linearity:} Note for any $a,b\in\K$ we have $\rho(aj_1+bj_2)=\int \ell_G{(aj_1+bj_2)}\;d\lambda=\int a\ell_G{(j_1)}+b\ell_G{(j_2)}\;d\lambda=a\int \ell_G{(j_1)}\;d\lambda+b\int\ell_G{(j_2)}\;d\lambda=a\rho(j_1)+b\rho(j_2)$.\n
 This reasoning uses property (1) of Remark \ref{liftingopproperties}, which follows from the deduction that for all $x\in G$ and $k_1,k_2\in R(H,\K)$ we have 
\[\begin{split}  
 \ell_G(ak_1+bk_2)(x)=&\begin{cases}
                        ak_1(x)+bk_2(x) & \mbox{if} \ x\in H\\
                        0& \mbox{if} \  x\in G\setminus H
                        \end{cases}\\
                        =&a\begin{cases}
                        k_1(x)& \mbox{if} \  x\in H\\
                        0&x\in G\setminus H
                        \end{cases}+b\begin{cases}
                        k_2(x) & \mbox{if} \  x\in H\\
                        0& \mbox{if} \ x\in G\setminus H
                        \end{cases}\\
                    &=a\ell_G(k_1)(x)+b\ell_G(k_2)(x). 
 \end{split}    
\] 
 The properties below of positivity, left invariance, and normalization also rely on the assertions of Remark \ref{liftingopproperties} which may  be  proved in a similar way.
 \item \textbf{Positivity:} Note for $j>0$ we have $\rho(j)=\int \ell_G({j})\;d\lambda$. This is positive as $\ell_G({j})>0$ (Remark \ref{liftingopproperties} property (2)) and Haar integrals are positive.
 
 \item \textbf{Left Invariance:} Note that for $h\in H$ we have $\rho({}_hj)=\int \ell_G{({}_hj)}\;d\lambda=\int {}_h\ell_G({j})\;d\lambda=\int \ell_G({j})\;d\lambda$ (Remark \ref{liftingopproperties} property (3))
 
 \item \textbf{Normalization:} Note that (given \textbf{1}$:H\rightarrow\K$ the constant function on $H$ that takes the value 1 everywhere) $\rho($\textbf{1}$)=\int \ell_G($\textbf{1}$)\;d\lambda=\int \chi_H\;d\lambda=1$
 
 \end{enumerate}
 Next, to show that $||\ell_G({j})||_2=||j||_{2,H}$ we note that (from Remark \ref{liftingopproperties} (4-5))
\[||\ell_G({j})||_2=\sqrt{\int \ell_G({j})\overline{\ell_G({j})}d\lambda}
 =\sqrt{\int |\ell_G({j})|^2d\lambda}
 =\sqrt{\int \ell_G(|j|^2)d\lambda}=\sqrt{\rho(|j|^2)} =\sqrt{\rho(j\;\overline{j})} =||j||_{2,H}.\]
\end{proof}

The main result is that that we can use $\rhgk$ to approximate elements in $\lck$.

\begin{theorem}[Approximations on $\lck$]\label{approxl2} If  $G$ is   a locally compact group with  a  nontrivial compact open subgroup $H$, then $\rhgk$ is a dense subspace of $\lck$.
\end{theorem}
\begin{proof} 
 To begin with, we let $\int\cdot \:d\mu$ be any Haar measure on $G$. This exists by Lemma \ref{resultofexistence}.  We will use this to construct an appropriate Haar measure for our scopes.  Fix a constant $a=\int \chi_H \;\;d\mu$, where $\chi_H$ is the characteristic function on $H$.  We then define a new Haar measure on $G$ given by
 \[ 
 \int\cdot \:d\lambda\coloneqq \frac{1}{a}\int\cdot \:d\mu.
 \] This Haar measure will measure the function $\chi_H$ to have the value of 1. We may think of $\lambda$ as having been normalised over $H$.
 
 \begin{figure}[!htbp]
 \centering
 \includegraphics[width=0.69 \linewidth]{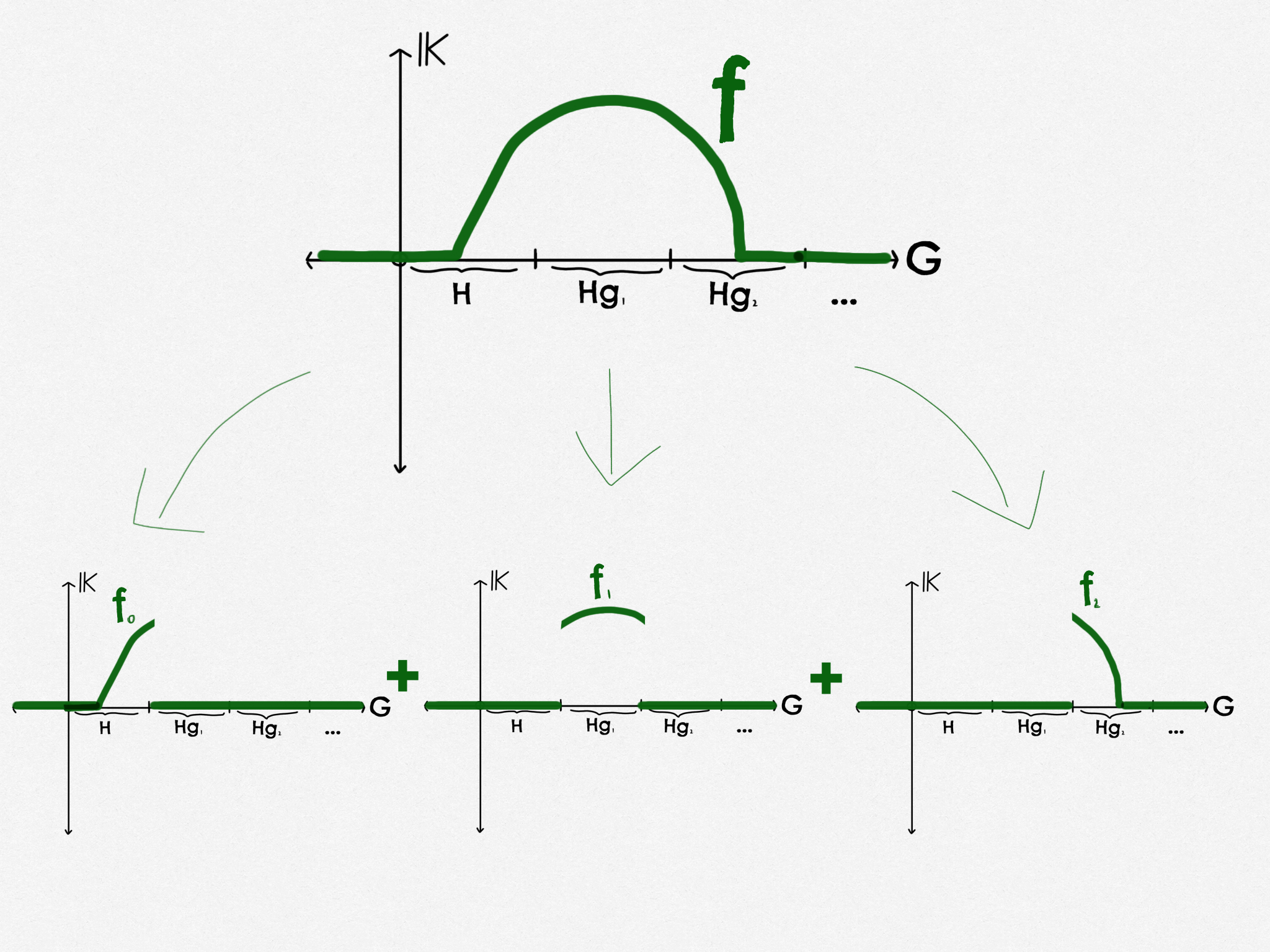}
 \caption{Demonstration of how the function is split into various sections-- Note that $G$ and $\K$ have been drawn like straight lines to give intuition, but will not in general have this structure}
 \label{SplittingF}
 \end{figure}
 
 Using this Haar measure we may construct the $L^2$-norm that we will be using for the remainder of this proof:
 \[ ||f||_2=\sqrt{\int f\bar{f}\: d\lambda}\]
 Let $f\in\cck$. We will find a sequence of functions in $\rhgk$ that approximate $f$ with respect to $||\cdot||_2$. In showing this, we will be able to conclude by the density of $\cck$ in $L^2(G,\K)$ that $\rhgk$ is dense in $L^2(G,\K)$. Let $G/H=\{Hg_i:i\in I\}$. Consider, for each $i\in I$
 \[
 f_i(x)=f(x)\chi_{Hg_i}(x)
\] Fig. \ref{SplittingF} shows how the function $f$ may be split into many of these $f_i$ functions. Since $f$ is assumed to be of compact support there will be finitely many nonzero $f_i$'s. This is because the cosets contained in $G/H$ form an open cover of the support of $f$ (in fact they form an open cover of $G$ itself) and we may therefore take a finite subcover of the support given without loss of generality by $\{Hg_1, Hg_2, ..., Hg_t\}$ ($t\in\N)$. 
 Thus $\mbox{supp}(f)\subseteq \bigcup_{i=1}^{t}Hg_i$ and  the nonzero $f_i$'s are just ${f_1, ..., f_t}$.
 Note that $$f(x)=\sum_{i=1}^t f_i(x)=\sum_{i=1}^{t}f(x)\chi_{Hg_i}(x).$$
 Each $f_i$ is continuous on $G$ as it is a product of continuous functions (with $\chi_{Hg_i}$ being continuous because $H$ is open). Our goal is to construct a sequence that will approximate each $f_i$, and then join those sequences' elements to construct a sequence that approximates $f$ itself. Choose a $f_i\in\{f_1, ..., f_t\}$. The support of this function is contained in $Hg_i$, so we may consider its entire nonzero portion (with possibly some of its zero portion) as a function on $Hg_i$. Since $Hg_i$ is topologically isomorphic to the compact group $H$, it is reasonable to assert that the function $f_i|_{Hg_i}$ may be approximated using the Peter-Weyl Theorem. However, the exact method used to do this is relatively delicate and bears explaining. The broad process is depicted in Fig. \ref{ApproximationAlgorithm}.
 \begin{figure}[!htbp]
 \centering
 \includegraphics[width=\linewidth]{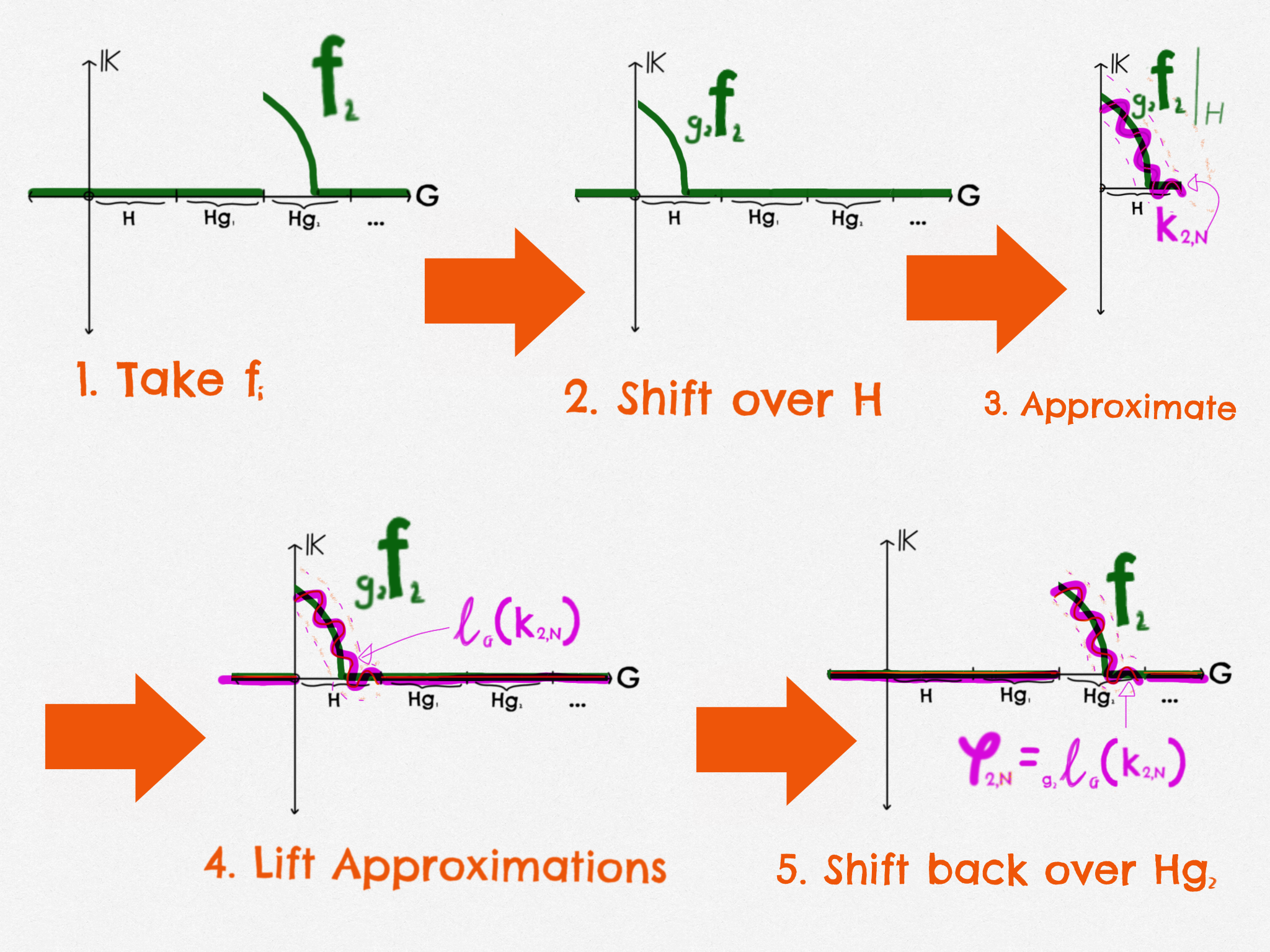}
 \caption{Illustration of the algorithms where the approximating lifted representative functions are constructed. Note that $G$ and $\K$ have been drawn like straight lines to give intuition, but will not in general have this structure.}
 \label{ApproximationAlgorithm}
 \end{figure}

 The idea is that we shift the function $f_i$ so that the portion defined on $Hg_i$ will be moved to be defined over $H$. This shift is given by ${}_{g_i}f_i$ (where as usual ${}_{g_i}f_i(x)=f_i(xg_i)$). See Step 2 of Fig. \ref{ApproximationAlgorithm} for this.  We then use the Peter-Weyl Theorem on ${}_{g_i}f_i|_{H}$ (as it is a continuous $H\rightarrow\K$ function) and obtain a sequence of functions $(k_{i,N})_{N=1}^{\infty}$ in $R(H,\K)$ that approximates it. This is depicted in Step 3 of Fig. \ref{ApproximationAlgorithm}. Specifically, we insist that
 \begin{equation}\label{inequalityinH}
 ||{}_{g_i}f_i|_{H}-k_{i,N}||_{2,H}<\frac{1}{N}
 \end{equation}
 Where the norm $||\cdot||_{2,H}$ is the $L^2$-norm produced using the normalised Haar measure on $H$. We need to use this norm here and not the norm $||\cdot||_2$ we gave above as we are performing approximations on $C(H,\K)$ and as such we need to use the norm appropriate to that space. They do, however, coincide to a certain extent due to how the Haar measure was normalized in the beginning.  We will see this later. What we do next is we lift those $k_{i,N}$ terms to be functions on $G$ using the lifting operation outlined in Definition \ref{liftingop}. They will in fact be lifted representative functions as per Definition \ref{liftedrep}. Owing to the fact that the Haar measure was normalized over $H$, we  obtain
 \begin{equation}\label{Gsqueeze}
 ||{}_{g_i}f_i-\ell_G({k_{i,N}})||_{2}<\frac{1}{N}
 \end{equation}
 This equation holds  because of Lemma \ref{liftinglemma}, following from \eqref{inequalityinH} and the fact that $\ell_G{({}_{g_i}f_i|_H-k_{i,N})}={}_{g_i}f_i-\ell_G({k_{i,N}})$.  Finally, we shift those $\ell_G({k_{i,N}})$ back over $Hg_i$ so that they approximate $f_i$ as seen in Step 4 and Step 5 of Fig. \ref{ApproximationAlgorithm}. For ease of notation, these approximating functions will be named $\varphi_{i,N}\coloneqq {}_{g_i}(\ell_G({k_{i,N}}))$ and will uphold the following equation coherently with  \eqref{Gsqueeze}, namely
 \[
 ||f_i-\varphi_{i,N}||_2=||f_i-{}_{g_i^{-1}}
 \ell_G({k_{i,N}})||_2=||{}_{g_i}f_i-
 \ell_G({k_{i,N}})||_2<\frac{1}{N}.
 \]
 The idea is  to use the sequences which we have just constructed
 \[(\varphi_{1,N})_{N=1}^{\infty}  \ \mbox{to approximate} \  f_1,  \ (\varphi_{2,N})_{N=1}^{\infty} \  \mbox{to approximate} \  f_2,  ...,  \ (\varphi_{t,N})_{N=1}^{\infty} \  \mbox{to approximate} \ f_t\]
in an appropriate way.  In fact we form a single larger sequence $(\varphi_N)_{N=1}^{\infty}$ given by sum $\varphi_N=\varphi_{1,N}+\varphi_{2,N}+...+\varphi_{t,N}$ where (since the sum is finite) each term  will  appear in the lifted representative functions as indicated in Fig. \ref{CompositeLiftedRep}. This sequence $(\varphi_{N})_{N=1}^{\infty}$ approximates $f$ as motivated by the arguments below
 \[ ||f-\varphi_N||_2 =||(f_1+f_2+...+f_t)-(\varphi_{1,N}+\varphi_{2,N}+...+\varphi_{t,N})||_2\]
\[ =||(f_1-\varphi_{1,N})+(f_2-\varphi_{2,N})+...+(f_t-\varphi_{t,N})||_2 \]
\[ \leq||f_1-\varphi_{1,N}||_2+||f_2-\varphi_{2,N}||_2+...+||f_t-\varphi_{t,N}||_2 <\frac{1}{N}+\frac{1}{N}+...+\frac{1}{N}=\frac{t}{N}.\] 
 Clearly as $N\rightarrow\infty$ we have that $\frac{t}{N}\rightarrow 0$ and $0 \leq ||f-\varphi_N||_2$ so by the Squeeze Theorem we can conclude that $\lim_{N\rightarrow\infty} \varphi_{N}=f$ with respect to the $L^2$-norm, as desired. So $\rhgk$ is dense in $\cck$. Since $\cck$ is dense in $\lck$ with respect to the $L^2$-norm, we have therefore constructed a dense subgroup of $\lck$. \end{proof}

We assumed $H$ to be nontrivial in Theorem \ref{approxl2}, on the basis of the significant evidences which can be found in the example below. On the other hand, if $H$ is trivial in Theorem \ref{approxl2}, then we end up in a situation where we are approximating functions defined on a discrete topology, with those of compact  support specifically only having finite support. As these topologies are not particularly interesting, we focused on large nontrivial compact open subgroups of locally compact groups.

\begin{example}A well known noncompact locally compact abelian group  is given by $\Qp$ which  may be constructed in many ways, but the method that is most illustrative to our purposes is its construction using inverse limits and the subgroup $\Zp$, as per \cite[Exercise E1.16]{hofmann}.  Just to report briefly this well known construction, we begin with  
 $\varphi_{n} \ : \  u+p^{n+1}\mathbb{Z} \in \mathbb{Z}(p^{n+1}) \longmapsto \varphi_{n}(u+p^{n+1}\mathbb{Z})= u+p^{n}\mathbb{Z} \in \mathbb{Z}(p^{n}),$
  surjective homomorphism  of finite $p$-groups  and get the inverse limit      $\mathbb{Z}_{p}=\underset{n\in\mathbb{N}}{\underleftarrow{\lim}} \ \mathbb{Z}(p^{n}).$ Then
$\Phi_{n} \ : \  v+p^{n+1}\mathbb{Z} \in \frac{\frac{1}{p^{\infty}}\mathbb{Z}}{p^{n+1}\mathbb{Z}} \longmapsto \Phi_{n}(v+p^{n+1}\mathbb{Z})= v+p^{n}\mathbb{Z} \in  \frac{\frac{1}{p^{\infty}}\mathbb{Z}}{p^{n}\mathbb{Z}}$   allows us to form the  corresponding inverse limit and we get \[
 \mathbb{Q}_p=\varprojlim_{n \in \mathbb{N}}   \  \ \frac{\frac{1}{p^\infty} \mathbb{Z} }{ p^n  \mathbb{Z}}=    \mathbb{Z}_p \cup \frac{1}{p} \mathbb{Z}_p \cup  \frac{1}{p^2} \mathbb{Z}_p \cup \ldots \cup \frac{1}{p^n} \mathbb{Z}_p \cup  \frac{1}{p^{n+1}} \mathbb{Z}_p \cup \ldots.
\]In this situation we have
 $\frac{1}{p^{\infty}}\mathbb{Z}=\bigcup_{n \in \mathbb{N}}\frac{1}{p^{n}}\mathbb{Z}\subseteq\mathbb{Q}$ and the quotient
 $ \frac{ \frac{1}{p^\infty}\mathbb{Z}}{\mathbb{Z}}=\mathbb{Z}(p^\infty)$, getting further quotients  $(\frac{ \frac{1}{p^\infty}\mathbb{Z}}{p^{n+1} \mathbb{Z}})/({\frac{ \mathbb{Z}}{p^{n+1}\mathbb{Z}}})   \simeq \mathbb{Z}(p^\infty)$
via inclusions
$\mathrm{incl}    :  a+p^{n+1}\mathbb{Z}   \mapsto a+p^{n+1}\mathbb{Z}$ and natural projections   $\mathrm{quot}  :       b+p^{n+1}\mathbb{Z} \mapsto b+\mathbb{Z}. $ Note that  the Pr\"ufer group is a discrete locally compact abelian noncompact group (direct limit of $\Z(p^n)$)
\[\mathbb{Z}(p^\infty)= \varinjlim_{n \in \mathbb{N}}  \ \mathbb{Z}(p^n). \]
The reader may refer to \cite[Definitions 1.25, 1.27, Lemma 1.26, Exercise E1.16, Example 1.28]{hofmann} for details. Under the topology of $\Qp$ which is given by the construction above,  $\Z_p$ is a nontrivial compact  open subgroup of $\Qp$ and we may apply Theorem \ref{approxl2} with $G=\Qp$ and $H=\Zp$. The following diagram is well known and applies to the present circumstances
\[\begin{CD}
\ldots @<<< 0 @<<< 0 @<<< 0 @<<< \ldots @<<< 0\\
@.  @VVV @VVV @VVV @. @VVV\\
\ldots @<<< \frac{ \mathbb{Z}}{p^{n-1}\mathbb{Z}} @<\varphi_{n-1}<< \frac{ \mathbb{Z}}{p^n\mathbb{Z}} @<\varphi_n<< \frac{ \mathbb{Z}}{p^{n+1}\mathbb{Z}} @<<< \ldots @<<< \mathbb{Z}_p \\
@. @VV_{\mathrm{incl}}V @VV_{\mathrm{incl}}V @VV_{\mathrm{incl}}V @. @VV_{\mathrm{incl}}V\\
\ldots @<<< \frac{ \frac{1}{p^\infty}\mathbb{Z}}{p^{n-1} \mathbb{Z}} @<\Phi_{n-1}<< \frac{ \frac{1}{p^\infty}\mathbb{Z}}{p^n \mathbb{Z}} @<\Phi_n<< \frac{ \frac{1}{p^\infty}\mathbb{Z}}{p^{n+1} \mathbb{Z}} @<<< \ldots @<<< \mathbb{Q}_p\\
@. @VV_{\mathrm{quot}}V @VV_{\mathrm{quot}}V @VV_{\mathrm{quot}}V @. @VV_{\mathrm{quot}}V\\
\ldots @<<< \mathbb{Z}(p^\infty) @= \mathbb{Z}(p^\infty) @= \mathbb{Z}(p^\infty) @= \ldots @= \mathbb{Z}(p^\infty) \\
@. @VVV @VVV @VVV @. @VVV\\
\ldots @<<< 0 @<<< 0 @<<< 0 @<<< \ldots @<<< 0\\
\end{CD}
\]
    It is worth noting that Theorem \ref{approxl2}  will never be useful in the context of a  noncompact connected locally compact group, due to the fact that a noncompact connected locally compact  group is poor of compact  open subgroups. In fact  such a subgroup together with its complement would form a disjoint open cover of the full group, contradicting its connectedness. In particular, we see this looking at the closed subgroups of $\mathbb{R}^n$. From \cite[Theorem A1.12]{hofmann}, we see that every closed subgroup of $\R^n$ is of the form $$\R\cdot e_1\oplus ...\oplus\R\cdot e_p\oplus\Z\cdot e_{p+1}\oplus...\oplus\Z\cdot e_{p+q}$$ where $\oplus$ denotes a direct product of groups and $e_1,...,e_{p+q}$ are some basis vectors of $\R ^n$ with $p,q\in\N$. No subgroup of this form aside from the trivial subgroup (which is not open) may ever be compact because it is not bounded (recall compact subsets of a Hausdorff space such as $\R^n$ are exactly those sets that are closed and bounded). This is an example of a connected group that does not satisfy the premises for our result.

Finally, an intermediate situation is offered by $G=\T \times \Qp$, where $\T \times \Zp$ is a nontrivial compact open subgroup in  $G$, which is a noncompact nonconnected locally compact abelian group. Here Theorem \ref{approxl2} also applies in a significant way.
\end{example}

%\textbf{Data availability.} Data sharing is not applicable to this article as no datasets were generated or analysed during the current study.

%\medskip

%\textbf{Declarations on conflict of interest.} On behalf of all authors, the corresponding author states that there is no conflict of interest.

\end{document}